\documentclass[reqno,12pt,centertags]{article}
\usepackage{amsmath,amsthm,amscd,amssymb,latexsym,upref,stmaryrd}
\usepackage[cp1251]{inputenc}
\usepackage[english]{babel}
\usepackage{mathtext}
\usepackage{amsfonts}
\usepackage{graphicx}
\usepackage[numbers,sort&compress]{natbib}
\usepackage{amsmath}

\numberwithin{equation}{section}
\numberwithin{figure}{section}
\usepackage{hyperref}
\newcommand*{\mailto}[1]{\href{mailto:#1}{\nolinkurl{#1}}}
\usepackage{pgfplots}
\usepackage{amsmath}
\usepackage{amsthm}
\newtheorem{corollary}{Corollary}
\newtheorem{theorem}{Theorem}
\newtheorem{lemma}{Lemma}
\newtheorem{remark}{Remark}
\theoremstyle{proposition}

\begin{document}
\thispagestyle{empty}

{\center \large\bf Direct and inverse spectral problems for the Schr\"{o}dinger operator with double generalized Regge boundary conditions}
\\[0.5cm]
\noindent {\bf  Xiao-Chuan Xu}\footnote{School of Mathematics and Statistics, Nanjing University of Information Science and Technology, Nanjing, 210044, Jiangsu,
People's Republic of China, {\it Email:
xcxu@nuist.edu.cn}}
\noindent {\bf and Yu-Ting Huang}\footnote{School of Mathematics and Statistics, Nanjing University of Information Science and Technology, Nanjing, 210044, Jiangsu,
People's Republic of China, {\it Email: hyt00142024@163.com}}
\\

\noindent{\bf Abstract.}
{In this paper, we study the direct and inverse spectral problems for the Schr\"{o}dinger operator with two generalized Regge boundary conditions. For the direct problem, we give the properties of the spectrum, including the asymptotic distribution of the eigenvalues. For the inverse problems, we prove several uniqueness theorems,  including the cases: even potential, two-spectra, as well as the general partial inverse problem.}

\medskip

\section{Introduction}

Consider the following generalized Regge problem $L\left(q,\alpha_{0},\beta_{0},\alpha,\beta\right)$
\begin{equation}\label{k1}
  \ell y:=-y''+q(x)y=\lambda^2 y ,\quad x\in(0,a),\\
  \end{equation}
  \begin{equation}\label{k2}
  \quad y'\left(0\right)-\left(i\alpha_{0}\lambda + \beta_{0}\right)y\left(0\right)=0,\\
  \end{equation}
 \begin{equation}\label{k3}
  \quad y'\left(a\right)+\left(i\alpha\lambda +\beta\right)y\left(a\right)=0,\\
\end{equation}
where the complex-valued potential function$q\in L^2(0,a)$, $\alpha_{0}\ge0,\alpha>0$ and $\beta_{0},\beta\in\mathbb{C}$.
The problem $L\left(q,\alpha_{0},\beta_{0},\alpha,\beta\right)$ arises in various mathematical and physical models. For example, the Liouville transformation can transform the problem of smooth inhomogeneous string vibration with viscous damping at both ends of the string into the problem $L\left(q,\alpha_{0},\beta_{0},\alpha,\beta\right)$ with $\alpha,\alpha_0>0$. When $\alpha_{0}=0$, it describes the problem of an inhomogeneous string vibration with no  damping at the left end but with various damping at the right end \cite{MV, mv}. The  problem $L\left(q,1,0,1,0\right)$ is the resonance scattering problem on the real line \cite{EK,ST,XF,Zo1}.

In the direct and inverse spectral theories, the Schr\"{o}dinger operator with boundary conditions independent of spectral parameters, i.e., the problem $L\left(q,0,\beta_{0},0,\beta\right)$, has the comparatively perfect researches (see the  monographs \cite{GV,Bb,V}). For the inverse spectral problem of $L\left(q,0,\beta_{0},0,\beta\right)$,  in general,  in order to recover the potential and the parameters in the boundary conditions, one needs to specify two spectra. Whereas, in some cases, one spectrum is enough, such as, (i) the even inverse problem \cite{GV}: $q(x)=q(a-x)$ and $\beta=\beta_0$, and (ii) the half-inverse problem \cite{HB}: $\beta$ is given and $q(x)$ is known a priori on $(\frac{a}{2},a)$. After the half-inverse problem, many authors also continued to study the general partial inverse problems, i.e.,  the potential is given on a subinterval $(b,a)$ with arbitrary $b\in (0,a)$. The uniqueness theorems of the general partial inverse problems were proved in \cite{DGS,FB,M,MH} and other works. One can refer to the review \cite{BP} for the general summaries of partial inverse problems.

When the boundary conditions contain the spectral parameter $\lambda$, the results of the problem will change. There are many works related Schr\"{o}dinger operator with the boundary condition dependent on $\lambda^2$ (see, e.g., \cite{BN1,BN2,FY0,G,G1,WW} and the references therein).
The dependency form of spectral parameters in the boundary condition \eqref{k2} or \eqref{k3}  is different from the works mentioned above. The boundary condition \eqref{k2} or \eqref{k3} is called the generalized Regge boundary condition. This kind of boundary condition has attracted a lot of attentions  from many scholars (see, e.g., \cite{BY,BS,WY,mv,RS,VC,XP,x,X,Y}). When there is only one generalized Regge boundary condition, one spectrum uniquely recovers all the unknown coefficients \cite{mv,VC,XP,EK0,RS,x,X,Y}. When there are two  generalized Regge boundary conditions, the situation becomes different, in particular, one spectrum is not enough to uniquely recover all the unknown coefficients. As far as we know, few ones considered two generalized Regge boundary conditions \eqref{k2} and \eqref{k3} directly  for the Schr\"{o}dinger operator. Whereas, these two boundary conditions were studied by some scholars for the  differential pencil  (see \cite{BY,BS,WY})
 \begin{equation}\label{hx9}
   -y''+(2\lambda p(x)+q(x))y=\lambda^2 y.
 \end{equation}
In \cite{BY,WY}, the two-spectra theorem is proved for the problem \eqref{k2}-\eqref{hx9}:  the spectrum of the problem \eqref{k2}-\eqref{hx9} and the spectrum of the problem \eqref{k3}, \eqref{hx9} and $y(0)=0$ uniquely determine $p$ and $q$ as well as all the unknown parameters in \eqref{k2} and \eqref{k3}. The uniqueness theorem of the half inverse problem was proved in \cite{BS}. 
Although \eqref{k1} is a particular case of \eqref{hx9}, the two-spectra theorem in \cite{BY,WY} is not applicable for the problem $L\left(q,\alpha_0,\beta_{0},\alpha,\beta\right)$ considered here. Because the spectrum of the problem \eqref{k3}, \eqref{hx9} and $y(0)=0$ with $p=0$ is enough to uniquely determine all the unknowns \cite{VC,X}. Hence, the two-spectra theorem for the problem $L\left(q,\alpha_0,\beta_{0},\alpha,\beta\right)$ should be reformulated and proved. Moreover, there is no result for the even inverse problem or the general partial inverse problem of the problem $L\left(q,\alpha_0,\beta_{0},\alpha,\beta\right)$ with $(1-\alpha)(\alpha_0-1)\ne0$, which are also interesting and important issues and should be investigated.

In this  paper, we study the direct and inverse spectral problems for the boundary value problem $L\left(q,\alpha_{0},\beta_{0},\alpha,\beta\right)$. For the direct problem, we give some properties of the eigenvalues, in particular, give the  asymptotic behavior. For the inverse spectral problem, we prove four uniqueness theorems, including the two-spectra theorem, even inverse problem as well as the general partial inverse problem.


This paper is structured as follows. In the second section, we introduce three characteristic functions. In the third section, we investigate the asymptotic distribution of eigenvalues and the fundamental properties of eigenvalues in the lower plane. In the last section, we consider the inverse problems and prove the uniqueness theorems.

\section{Characteristic functions}
In this section, we introduce three characteristic functions.
 Let $s\left(\lambda,x\right)$ and $c\left(\lambda,x\right)$ be the solutions of the initial value problems for \eqref{k1} with the initial conditions $s\left(\lambda,0\right)=0$, $s'\left(\lambda,0\right)=1$ and $c\left(\lambda,0\right)=1$, $c'\left(\lambda,0\right)=\beta_{0}$, respectively.
It was shown in \cite[p.9]{V} that
\begin{equation}\label{k4}
    s\left(\lambda,x\right)=\frac{\sin{\lambda x}}{\lambda}+\int_{0}^{x}K_{1}\left(x,t\right)\frac{\sin{\lambda t}}{\lambda}dt,
\end{equation}
\begin{equation}\label{k5}
    c\left(\lambda,x\right)=\cos{\lambda x}+\int_{0}^{x}G\left(x,t,\beta_{0}\right)\cos{\lambda t}dt,
\end{equation}
where
\begin{equation}\label{k6}
    K_{1}\left(x,t\right)=K\left(x,t\right)-K\left(x,-t\right),
\end{equation}
\begin{equation}\label{k6}
    G\left(x,t,\beta_{0}\right)=\beta_{0}+K\left(x,t\right)+K\left(x,-t\right)+\beta_{0}\int_{t}^{x}[K\left(x,\xi\right)-K\left(x,-\xi\right)]d\xi,
\end{equation}
and $K\left(x,t\right)$ is the unique solution of the integral equation

\begin{equation}
\notag
    K\left(x,t\right)=\frac{1}{2}\int_{0}^{\frac{x+t}{2}}q(s)ds+
    \int_{0}^{\frac{x+t}{2}}\int_{0}^{\frac{x-t}{2}}q\left(s+v\right)K\left(s+v,s-v\right)dv ds,
\end{equation}
in the region $\left\{\left(x,t\right)\in (0,a)\times(0,a):\lvert t \rvert\leq x \right\}$. Moreover,
\begin{equation}\label{mza}
    G\left(x,x,\beta_0\right)=\beta_0+K_{1}\left(x,x\right),\quad K_{1}\left(x,x\right)=K\left(x,x\right)=\frac{1}{2}\int_{0}^{x}q\left(t\right)dt.
\end{equation}
Let
\begin{equation}\label{hx4}
        y\left(\lambda,x\right)=c\left(\lambda,x\right)+i\alpha_{0}\lambda s\left(\lambda,x\right).
\end{equation}
Then
\begin{equation}\label{r2}
    y\left(\lambda,x\right)=\cos{\lambda x}+i\alpha_{0}\sin{\lambda x}+G(x,x,\beta_{0})\frac{\sin{\lambda x}}{\lambda}-i\alpha_{0}K_{1}\left(x,x\right)\frac{\cos{\lambda x}}{\lambda}+\frac{\Psi_{1}\left(\lambda,x\right)}{\lambda},
\end{equation}
and
\begin{equation}\label{k8}
        y'\left(\lambda,x\right)=-\lambda \sin{\lambda x}+\left[G(x,x,\beta_{0})+i\alpha_{0}\lambda\right]\cos{\lambda x}+i\alpha_{0}K_{1}\left(x,x\right)\sin{\lambda x}+\Psi_{2}\left(\lambda,x\right),
\end{equation}
where
\begin{equation}\label{hx}
    \Psi_{1}\left(\lambda,x\right)=-\int_{0}^{x}G_{t}\left(x,t,\beta_{0}\right)\sin{\lambda t}dt+i\alpha_{0}\int_{0}^{x}K_{1t}\left(x,t\right)\cos{\lambda t}dt,
\end{equation}
and
\begin{equation}
    \Psi_{2}\left(\lambda,x\right)=\int_{0}^{x}G_{x}\left(x,t,\beta_{0}\right)\cos{\lambda t}dt+i\alpha_{0}\int_{0}^{x}K_{1x}\left(x,t\right)\sin{\lambda t}dt.
\end{equation}
 The spectrum of problem $L\left(q,\alpha_{0},\beta_{0},\alpha,\beta\right)$ coincides with the sets of zeros of the entire function
\begin{equation}\label{k9}
    \Delta_+\left(\lambda\right)=y'\left(\lambda,a\right)+\left(i\alpha\lambda+\beta\right)y\left(\lambda,a\right),
\end{equation}
 which is called the \emph{characteristic function} of the problem $L\left(q,\alpha_{0},\beta_{0},\alpha,\beta\right).$ Denote $\left<y,z\right>=yz'-y'z$. It is easy to show that if $y$ and $z$ are solutions of \eqref{k1}, then $\left<y,z\right>$ is independent of $x$. Let $\phi\left(\lambda,x\right)$ be the solution of the problem for \eqref{k1} with initial conditions $\phi\left(\lambda,a\right)=1$ and $\phi'\left(\lambda,a\right)=-\left(i\alpha\lambda+\beta\right).$ Then we have
 \begin{equation}\label{e10}
        \Delta_+\left(\lambda\right)=\left<\phi\left(\lambda,x\right),y\left(\lambda,x\right)\right>.
    \end{equation}
    Define
    \begin{equation}\label{e15}
        \Delta_{-}\left(\lambda\right)=y'\left(\lambda,a\right)-\left(i\alpha\lambda-\beta\right)y\left(\lambda,a\right).
    \end{equation}
     Obviously, $\Delta_{-}\left(\lambda\right)$ is the \emph{characteristic function} of the problem $L\left(q,\alpha_{0},\beta_{0},-\alpha,\beta\right)$, i.e., the equation \eqref{k1} with the boundary conditions \eqref{k2} and $$y'\left(\lambda,a\right)-\left(i\alpha\lambda-\beta\right)y\left(\lambda,a\right)=0.$$
  Using \eqref{k9} and \eqref{e15}, it is easy to verify
    \begin{equation}\label{e7}
        \Delta_+\left(\lambda\right)\Delta_+\left(-\lambda\right)-\Delta_{-}\left(\lambda\right)\Delta_{-}\left(-\lambda\right)=4\alpha\alpha_{0}\lambda^2.
    \end{equation}
     We also consider the function
    \begin{equation}\label{o30}
        \Delta_{0}(\lambda)=y\left(\lambda,a\right).
    \end{equation}
It is easy to show that the function $\Delta_{0}(\lambda)$ is the \emph{characteristic function} of the following problem
\begin{equation}\label{hx12}
    -y''+q_{1}(x)y=\lambda^2 y ,\quad x\in(0,a),
\end{equation}
\begin{equation}\label{hx13}
    y(0)=0,
\end{equation}
\begin{equation}\label{hx14}
    y'(a)+(i\alpha_{0}\lambda+\beta_{0})y(a)=0,
\end{equation}
where $q_{1}(x)=q(a-x)$.

 \section{Properties of eigenvalues}
 In this section, we investigate the properties of the spectrum of the problems $L\left(q,\alpha_{0},\beta_{0},\pm\alpha,\beta\right)$. Let's first give the asymptotic distribution of the eigenvalues.
\begin{theorem}\label{th1}
     $\left(1\right)$ If $\left(\alpha_{0}-1\right)\left(1-\alpha\right)>0$, then the eigenvalues $\{\lambda_{k}^\pm\}_{k\in\mathbb{Z}}$ of the problems $L\left(q,\alpha_{0},\beta_{0},\pm\alpha,\beta\right)$, respectively,  have the following asymptotic behavior:
   \begin{equation}\label{k11}
   \begin{aligned}
      \lambda_{k}^\pm=\frac{\pi}{a}\left(\lvert k\rvert-\frac{1}{2}\right){\rm sgn} k+\frac{i P_0^\pm}{2a}+\frac{P}{k}+\frac{\beta_{k}}{k}, \quad \lvert k \rvert \to \infty,
   \end{aligned}
   \end{equation}
   where $\{\beta_{k}\}_{k=-\infty,k\neq0}^{\infty}\in l_{2}$ and where
   \begin{equation}\label{k10}
   \begin{aligned}
       P_0^\pm=\ln{\left(\frac{|\alpha_{0}\pm\alpha+1\pm\alpha\alpha_{0}|}{|\alpha_{0}\pm\alpha-(1\pm\alpha\alpha_{0})|}\right)},\quad P=\frac{\beta_0}{\pi\left(1-\alpha_{0}^2\right)}+\frac{\beta}{\pi\left(1-\alpha^2\right)}+\frac{K_{1}\left(a,a\right)}{\pi}.
   \end{aligned}
   \end{equation}
   $\left(2\right)$ If $\left(\alpha_{0}-1\right)\left(1-\alpha\right)<0$, then the eigenvalues $\{\lambda_{k}^\pm\}_{k\in\mathbb{Z}_{0}}$, $\mathbb{Z}_{0}=\mathbb{Z}\setminus\{0\}$, of the problem $L\left(q,\alpha_{0},\beta_{0},\pm\alpha,\beta\right)$, respectively, have the following asymptotic behavior:
   \begin{equation}\label{k12}
   \begin{aligned}
        \lambda_{k}^\pm=\frac{\pi}{a}\left(\lvert k\rvert-1\right){\rm sgn}  k+\frac{iP_0^\pm}{2a}+\frac{P}{k}+\frac{\beta_{k}}{k}, \quad \lvert k \rvert \to \infty,
   \end{aligned}
   \end{equation}
   where $P_0^\pm$ and $P$ are given \eqref{k10} and $\{\beta_{k}\}_{k=-\infty,k\neq0}^{\infty}\in l_{2}$.\\
   $\left(3\right)$ If $\left(\alpha_{0}-1\right)\left(1-\alpha\right)=0$, then the problem $L\left(q,\alpha_{0},\beta_{0},\pm\alpha,\beta\right)$ may have only finitely many eigenvalues. If there exist infinitely many eigenvalues $\{\lambda_{k}^+\}$, then $Im \lambda_{k}^+ \to \infty$ as $\lvert k\rvert\to \infty$.
\end{theorem}
\begin{proof}
    Using \eqref{r2}-\eqref{k8} we obtain \begin{align}\label{c16}
        \Delta_\pm\left(\lambda\right)=&\lambda[i\left(\alpha_{0}\pm\alpha\right)\cos{\lambda a}-\left(1\pm\alpha\alpha_{0}\right)\sin{\lambda a}]+i[\alpha_{0}\beta\pm\alpha \omega+\alpha_{0} K_{1}\left(a,a\right)]\sin{\lambda a}\\
        \notag &+[\omega+\beta\pm\alpha_{0}\alpha K_{1}\left(a,a\right)]\cos{\lambda a}+\psi_\pm\left(\lambda\right).
    \end{align}
     where $\omega=G(a,a,\beta_{0})$ and
    \begin{equation}\label{o80}
        \psi_\pm\left(\lambda\right)=\omega\beta\frac{\sin{\lambda a}}{\lambda}-i\beta\alpha_{0}K_{1}(a,a)\frac{\cos{\lambda a}}{\lambda}+(\beta \pm i\alpha\lambda)\frac{\Psi_{1}\left(\lambda,a\right)}{\lambda}+\Psi_{2}\left(\lambda,a\right).
    \end{equation}
It is obvious  that the function $\psi_\pm(\lambda)$ are entire functions of exponential type $\le a$ and belong to $L^2(-\infty,\infty)$.
Rewrite \eqref{c16} as
    \begin{align}\label{k13}
     \notag   -i\Delta_\pm\left(\lambda\right)=&\lambda[\left(\alpha_{0}\pm\alpha\right)\cos{\lambda a}+i\left(1\pm\alpha\alpha_{0}\right)\sin{\lambda a}]+[\alpha_{0}\beta\pm\alpha \omega+\alpha_{0} K_{1}\left(a,a\right)]\sin{\lambda a}\\
       &-i[\omega+\beta\pm\alpha_{0}\alpha K_{1}\left(a,a\right)]\cos{\lambda a}-i\psi_\pm\left(\lambda\right).
    \end{align}
The right hand side of \eqref{k13} is of the form \eqref{A.7} in Lemma \ref{LA2} with $M^\pm=\alpha_{0}\beta\pm\alpha \omega+\alpha_{0} K_{1}\left(a,a\right)$ and $N^\pm=\omega+\beta\pm\alpha_{0}\alpha K_{1}\left(a,a\right)$, and the number $P$ defined in \eqref{A.9} becomes
    \begin{equation}
        P^\pm=\frac{\left(1\pm\alpha_{0}\alpha\right)N^\pm-\left(\alpha_{0}\pm\alpha\right)M^\pm}{\pi[\left(1\pm\alpha_{0}\alpha\right)^2-\left(\alpha_{0}\pm\alpha\right)^2]}
        =\frac{(1-\alpha^2)\omega+(1-\alpha_0^2)\beta+(\alpha^2-1)\alpha_0^2K_1(a,a)}{\pi(1-\alpha_0^2)(1-\alpha^2)}.
    \end{equation}
  Using Lemma \ref{LA2} and noting $\omega=\beta_0+K_1(a,a)$, we obtain \eqref{k11}-\eqref{k12}.

  When $\left(\alpha_{0}-1\right)\left(1-\alpha\right)=0$, we have $\alpha_0=1$ or $\alpha=1$. In the case $\alpha_{0}=1$, we have
    \begin{equation}\label{o26}
        \Delta_\pm(\lambda)\!=\!\left[\!(1\pm\alpha)\!\left(i\lambda+\frac{\omega}{2}+\frac{K_{1}(a,a)}{2}\!\right)\!+\beta\!\right]e^{i\lambda a}+\frac{1\mp\alpha}{e^{i\lambda a}}\!\left[\frac{\omega}{2}-\frac{K_{1}(a,a)}{2}\right]\!+\psi_\pm(\lambda).
    \end{equation}
    When $q=0$ and $\beta_{0}=0$ in \eqref{o26}, we have
    \begin{equation}
        \Delta_\pm(\lambda)=[i\lambda(1\pm\alpha)+\beta]e^{i\lambda a}.
    \end{equation}
    It implies that the unique zero of $\Delta_+(\lambda)$ is $\frac{i\beta}{1+\alpha}$ and $\Delta_-(\lambda)$ also has the only zero $\frac{i\beta}{1-\alpha}$ if $\alpha\ne1$. If $\Delta_+(\lambda)$ has infinitely many eigenvalues $\{\lambda_{k}^+\}$, then $\lvert\lambda_{k}^+\rvert\to\infty$ as $\lvert k\rvert\to\infty.$  Substituting $\lambda=\lambda_{k}^+$ into \eqref{o26}, respectively, and noting that $\psi(\lambda_k^+)=o(e^{|{\rm Im}\lambda_k^+|a})$, then we have
    \begin{equation}\label{o27}
        (1+\alpha)\!\!\left[i\lambda_{k}^++\frac{\omega}{2}+\frac{K_{1}(a,a)}{2}\right]+\beta\!=\!\frac{1-\alpha}{e^{2i\lambda_{k}^+ a}}
        \!\left[\!\frac{K_{1}(a,a)}{2}-\frac{\omega}{2}\!\right]\!\!+o(1)e^{\left(\lvert {\rm Im}\lambda_{k}^+\rvert+{\rm Im}\lambda_{k}^+\right)a}.
    \end{equation}
    If $ {\rm Im}\lambda_{k}^+$ is bounded or goes to $-\infty$ as $\lvert k\rvert\to\infty$, then the left hand side of \eqref{o27} is unbounded and the right hand is bounded, which is a contradiction. It shows that $\lim_{k\to\infty}\sup {\rm Im}\lambda_{k}^+=\infty.$ The left hand side of \eqref{o27} tending to $\infty$ implies for the right hand side that $2{\rm Im}\lambda_{k}^+\to\infty$ or $\lvert {\rm Im}\lambda_{k}^+\rvert+{\rm Im}\lambda_{k}^+\to\infty$ as $k\to\infty$, which gives ${\rm Im}\lambda_{k}^+\to\infty$. Similarly, we can also obtain the same results in the case $\alpha=1$.
The proof is complete.
\end{proof}
\begin{theorem}\label{th2x}
   (1) There are at most a finite number of the eigenvalues of  $L\left(q,\alpha_{0},\beta_{0},\alpha,\beta\right)$ lying in the closed lower half-plane. If $(\alpha_0-1)(1-\alpha)\ne0$ and $(\alpha_0+1)|\alpha-1|>(\alpha+1)|\alpha_0-1|$ ($(\alpha_0+1)|\alpha-1|<(\alpha+1)|\alpha_0-1|$), then there are at most a finite number of the eigenvalues of  $L\left(q,\alpha_{0},\beta_{0},-\alpha,\beta\right)$ lying in the closed lower (upper) half-plane. \\
    (2) If $q(x)$ is a real-valued function and $\beta,\beta_{0}\in\mathbb{R}$, then the eigenvalues of the problems $L\left(q,\alpha_{0},\beta_{0},\pm\alpha,\beta\right)$ have the following properties:\\
    (i) If $0$ is an eigenvalue of the problem $L\left(q,\alpha_{0},\beta_{0},\alpha,\beta\right)$, then it is simple.\\
    (ii) If $\lambda$ is an eigenvalue of the problem $L\left(q,\alpha_{0},\beta_{0},\pm\alpha,\beta\right)$, then $-\overline{\lambda}$ is also an eigenvalue, where $\overline{\lambda}$ is the complex conjugate of $\lambda$.\\
    (iii)  All nonzero eigenvalues of the problem $L\left(q,\alpha_{0},\beta_{0},\alpha,\beta\right)$ in the lower half-plane lie on the negative imaginary semiaxis $i\mathbb{R}_-$ and are simple.\\
    (iv) Assume that there are  $\kappa$ eigenvalues of $L\left(q,\alpha_{0},\beta_{0},\alpha,\beta\right)$ on $i\mathbb{R}_-$, denoted by $\lambda_{-j}=-i\lvert \lambda_{-j} \rvert$,$j=1,\ldots,\kappa$, satisfying $\lvert \lambda_{-j} \rvert<\lvert \lambda_{-(j+1)}\rvert$. Then $$i\dot{\Delta}_+\left(-i\lvert\lambda_{-j}\rvert\right)(-1)^{\kappa-j}<0,\quad
    \Delta_{0}\left(-i\lvert\lambda_{-j}\rvert\right)(-1)^{\kappa-j}>0,\quad j=1,\ldots,\kappa,$$ where $\dot{\Delta}_+(\lambda)=\frac{d {\Delta_+}(\lambda)}{d\lambda}$. Moreover, the zeros of $\Delta_{0}(\lambda)$ and $\Delta_+(\lambda)$ on $i\mathbb{R}_-$  interlace each other.
\end{theorem}
\begin{proof}
    (1) From Theorem \ref{th1}, it is easy to see that there are at most a finite number of the eigenvalues $\{\lambda_n^+\}$ lying in the closed lower half-plane. Assume that $(\alpha_0-1)(1-\alpha)\ne0$, (i) if $(\alpha_0+1)|\alpha-1|>(\alpha+1)|\alpha_0-1|$, then $\frac{|\alpha_{0}-\alpha+1-\alpha\alpha_{0}|}{|\alpha_{0}-\alpha-(1-\alpha\alpha_{0})|}>1$, which implies  $P_0^->0$; (ii) if  $(\alpha_0+1)|\alpha-1|<(\alpha+1)|\alpha_0-1|$, then $\frac{|\alpha_{0}-\alpha+1-\alpha\alpha_{0}|}{|\alpha_{0}-\alpha-(1-\alpha\alpha_{0})|}\in (0,1)$, which implies   $P_0^-<0$.

    (2) (i)
    Due to \eqref{k1}, then the differentiation of \eqref{k1} with respect to $\lambda$ gives
    \begin{equation}\label{e12}
        -\dot{y}''\left(\lambda,x\right)+q\left(x\right)\dot{y}\left(\lambda,x\right)=2\lambda y\left(\lambda,x\right)+\lambda^2\dot{y}\left(\lambda,x\right).
    \end{equation}
   Multiply \eqref{k1} and \eqref{e12} by $\dot{y}$ and $y$, respectively, take the difference, and get
   \begin{equation}\label{e5s}
       2\lambda y^2\left(\lambda,x\right)
       =\left[y'\left(\lambda,x\right)\dot{y}\left(\lambda,x\right)-\dot{y}'\left(\lambda,x\right)y\left(\lambda,x\right)\right]'.
   \end{equation}
   Let us integrate both side of \eqref{e5s}, and using \eqref{k2}, then we have
   \begin{align}
   \notag
       2\lambda\int_{0}^{a}y^2\left(\lambda,x\right)dx&=\int_{0}^{a}\left[y'\left(\lambda,x\right)\dot{y}\left(\lambda,x\right)-\dot{y}'\left(\lambda,x\right)y\left(\lambda,x\right)\right]'dx\\
       &\notag=y'\left(\lambda,a\right)\dot{\Delta}_{0}\left(\lambda\right)-\dot{y}'\left(\lambda,a\right)\Delta_{0}\left(\lambda\right)+i\alpha_{0}\\
       &\label{o1}=\Delta_+\left(\lambda\right)\dot{\Delta}_{0}\left(\lambda\right)-\dot{\Delta}_+\left(\lambda\right)\Delta_{0}\left(\lambda\right)+i\alpha \Delta_{0}^{2}(\lambda)+i\alpha_{0}.
   \end{align}
   Letting $\lambda=0$ in \eqref{o1}, due to $\Delta_+(0)=0$, we have
   \begin{equation}
       \dot{\Delta}_+\left(0\right)\Delta_{0}\left(0\right)=i\alpha \Delta_{0}^{2}(0)+i\alpha_{0},
   \end{equation}
   which implies
   \begin{equation}
       -i\dot{\Delta}_+\left(0\right)\Delta_{0}\left(0\right)=\alpha \Delta_{0}^{2}(0)+\alpha_{0}>0.
   \end{equation}
   Hence, if $0$ is an eigenvalue, then it is simple.

  (ii)    Note that
    \begin{equation*}
        \overline{ c^{(\nu)}\left(\lambda,x\right)}=c^{(\nu)}\left(-\bar{\lambda},x\right), \quad  \overline{ s^{(\nu)}\left(\lambda,x\right)}=s^{(\nu)}\left(-\bar{\lambda},x\right),\quad \nu=0,1.
    \end{equation*}
 It follows from \eqref{hx4}  that
        \begin{align}\label{k16}
            y^{(\nu)}\left(-\overline{\lambda},x\right)= \overline{y^{(\nu)}\left(\lambda,x\right)},\quad \nu=0,1,
        \end{align}
which implies from \eqref{k9} and \eqref{e15} that
    \begin{equation}\label{m1s}
            \overline{\Delta_\pm \left(\lambda\right)}=\Delta_\pm\left(-\overline{\lambda}\right).
    \end{equation}
The assertion in (i) follows from \eqref{m1s}.

 (iii)   Assume that there is an eigenvalue $\lambda_{0}=\sigma-i\tau\left(\tau>0\right)$, i.e.,  $\Delta_+\left(\lambda_{0}\right)=0$. Using the initial condition of $y(\lambda_0,x)$ and integration by the parts, we  calculate
    \begin{align}\label{e1}
        \notag &\lambda_{0}^2\int_{0}^{a}\lvert y\left(\lambda_{0},x\right)\rvert^2dx=
             \int_{0}^{a}\ell y(\lambda_0,x)y\left(-\overline{\lambda}_{0},x\right)dx\\
           \notag= &i\alpha\left(\lambda_{0}+\overline{\lambda}_{0}\right)\lvert y\left(\lambda_{0},a\right)\rvert^2+i\alpha_0\left(\lambda_{0}+\overline{\lambda}_{0}\right)+ \int_{0}^{a} y(\lambda_0,x)\ell y\left(-\overline{\lambda}_{0},x\right)dx\\
           =& 2i\sigma\left[\alpha\lvert y(\lambda_{0},a)\rvert^2+\alpha_{0}\right]+
          \notag \overline{\lambda}_{0}^2\int_{0}^{a}\lvert y\left(\lambda_{0},x\right)\rvert^2dx,
    \end{align}
which implies
    \begin{equation}\label{e3}
       - 2\sigma\tau\int_{0}^{a}\lvert y\left(\lambda_{0},x\right)\rvert^2dx=\alpha\sigma\lvert y(\lambda_{0},a)\rvert^2+\sigma\alpha_{0}.
    \end{equation}
    By observing \eqref{e3}, we have that $\tau<0$ if $\sigma\neq0$, which contradicts the fact that $\tau>0$. Thus $\sigma=0$.

    Next, we verify that the eigenvalues $\lambda_{0}=-i\tau \left(\tau>0\right)$ is simple. Letting $\lambda=\lambda_{0}=-i\tau$ in \eqref{o1}, due to $\Delta_+\left(-i\tau\right)=0$, we have
   \begin{align}
       \dot{\Delta}_+\left(-i\tau\right)\Delta_{0}\left(-i\tau\right)=2i\tau\int_{0}^{a}y^2\left(-i\tau ,x\right)dx+i\alpha \Delta_{0}^{2}(-i\tau)+i\alpha_{0},
   \end{align}
which implies
\begin{equation}\label{o2}
    -i\dot{\Delta}_+\left(-i\tau\right)\Delta_{0}\left(-i\tau\right)>0.
\end{equation}
Therefore, if $\lambda_{0}$ is a zero of $\Delta_+(\lambda)$ on the negative imaginary semiaxis, then it is simple.

 (iv) Note that
   \begin{equation}\label{m1}
       \cos({-i\tau a})=\frac{e^{-\tau a}+e^{\tau a}}{2},  \sin({-i\tau a})=\frac{-e^{-\tau a}+e^{\tau a}}{2i}.
   \end{equation}
   Substituting \eqref{m1} into \eqref{c16}, then we have
   \begin{align}
       \Delta_+(-i\tau)\notag=&\left[\left(\alpha+\alpha_{0}\right)\frac{e^{-\tau a}+e^{\tau a}}{2}+\left(1+\alpha\alpha_{0}\right)\frac{-e^{-\tau a}+e^{\tau a}}{2}\right]\tau\\&\notag+\left[\omega+\beta+\alpha\alpha_{0}K_{1}\left(a,a\right)\right]\frac{e^{-\tau a}+e^{\tau a}}{2}\\
   \notag    &+\left[\alpha_{0}\beta+\alpha\omega+\alpha_{0}K_{1}\left(a,a\right)\right]\frac{-e^{-\tau a}+e^{\tau    a}}{2}+O\left(\frac{e^{\tau a}}{\tau}\right).
   \end{align}
    Then
    \begin{equation}\label{q9}
        \Delta_+(-i\tau)\to +\infty, \quad \tau\to +\infty.
    \end{equation}
    Note that $\Delta_+(\lambda)$ has only simple zeros on on the negative imaginary semiaxis. It follows from \eqref{q9} that $-i\dot{\Delta}_+\left(-i\lvert\lambda_{-\kappa}\rvert\right)>0$. Consequently, we have
   \begin{equation}
   \notag
       -i\dot{\Delta}_+\left(-i\lvert\lambda_{-j}\rvert\right)(-1)^{\kappa-j}>0.
   \end{equation}
   Together with \eqref{o2}, we have $   \Delta_{0}\left(-i\lvert\lambda_{-j}\rvert\right)(-1)^{\kappa-j}>0.$

   Substituting $\lambda=-i\tau$ into \eqref{o1}, we obtain
\begin{align}
\notag
  &  \Delta_+\left(-i\tau \right)\frac{\mathrm{d}}{\mathrm{d}(-i\tau)}\Delta_{0}\left(-i\tau\right)-\Delta_{0}\left(-i\tau\right)\frac{\mathrm{d}}{\mathrm{d} (-i\tau)}\Delta_+\left(-i\tau\right)\\&=-2i\tau\int_{0}^{a}\lvert y\left( -i\tau,x\right)\rvert^2dx-i\alpha \Delta_{0}^{2}(-i\tau)-i\alpha_{0}.
\end{align}
Consequently,
\begin{align}
\notag
    \frac{\mathrm{d}}{\mathrm{d}\tau }\left(\frac{\Delta_{0}\left(-i\tau\right)}{\Delta_+\left(-i\tau\right)}\right)&=\frac{1}{\Delta_+^2\left(-i\tau\right)}\left[2\tau\int_{0}^{a}\lvert y\left(-i\tau,x\right)\rvert^2dx+\alpha\Delta_{0}^{2}(-i\tau)+\alpha_{0}\right]>0,
\end{align}
thus the function $\frac{\Delta_{0}\left(-i\tau\right)}{\Delta_+\left(-i\tau\right)}$ monotonically increasing for $\tau\in\mathbb{R}^{+}\verb|\|\left\{ \lvert \lambda_{-j} \rvert|j=1,\ldots,\kappa\right\}$
with
\begin{align*}
   &\lim_{\tau\to \lvert \lambda_{-j} \rvert^\pm}\frac{\Delta_{0}\left(-i\tau\right)}{\Delta_+\left(-i\tau\right)}=\mp\infty,\quad j=2,\dots,\kappa-1\\
    &\label{e5}\lim_{\tau\to \lvert \lambda_{-1} \rvert^{+}}\frac{\Delta_{0}\left(-i\tau\right)}{\Delta_+\left(-i\tau\right)}=-\infty,\quad
    \notag\lim_{\tau\to \lvert \lambda_{-\kappa} \rvert^{-}}\frac{\Delta_{0}\left(-i\tau\right)}{\Delta_+\left(-i\tau\right)}=+\infty.
\end{align*}
Hence, $\Delta_{0}\left(-i\lvert \lambda_{-j} \rvert\right)\Delta_{0}\left(-i\lvert \lambda_{-(j+1)} \rvert\right)<0$. It follows that the function $\Delta_{0}(\lambda)$ has one zero in the interval $\left(-i\lvert \lambda_{-(j+1)} \rvert,-i\lvert \lambda_{-j} \rvert\right)$, then we have that the zeros of $\Delta_+(\lambda)$ and $\Delta_{0}(\lambda)$ interlace on the negative imaginary semiaxis.
The proof is complete.
\end{proof}
\section{Inverse problems}
In this section, we study the inverse spectral problems. We will consider the even inverse problem, two-spectra theorem and the general partial inverse problem, and prove four uniqueness theorems.

\begin{lemma}\label{hx11}
Let $f(z)$ be an   entire function of exponential type, and has the asymptotics
\begin{equation}\label{hx10}
 f(z)=z[c_1\cos z +c_2\sin z ]+O(e^{|{\rm Im}z|}),\quad |z|\to \infty,
\end{equation}
where $c_1$ and $c_2$ are constants. If $|c_1|^2+|c_2|^2\ne0$ and there exists one non-zero $c_0\in \{c_1,c_2,c_1+c_2,c_1-c_2\}$ and $c_0$ is known, then $f(z)$ is uniquely determined by all its zeros (including multiplicity).
\end{lemma}
\begin{proof}
Let $\{z_n\}$ (counted with multiplicities) be the non-zero zeros of the function $f(z)$. The Hadamard's factorization theorem implies
\begin{equation}\label{3.0}
f(z)=ce^{bz}E(z),\quad E(z):=z^{s}\prod_{z_n\ne0}\left(1-\frac{z}{z_n}\right)e^{\frac{z}{z_n}},
\end{equation}
where $b,c$ are the constants to be determined, and $s\ge 0$. Without loss of generality, assume $c_1\ne 0$ and is known. Then from (\ref{hx10}) we have
\begin{equation*}
  \frac{f(2n\pi)}{2n\pi}=c_1(1+o(1)),\quad n\to\infty,
\end{equation*}
which implies from (\ref{3.0}) that
\begin{equation}\label{3.0s}
  b=-\lim_{n\to\infty} \frac{\ln E(2n\pi)}{2n\pi},
\quad
  c= c_1\left[\lim_{n\to\infty}\frac{e^{2n\pi b}E(2n \pi)}{2n\pi}\right]^{-1}.
\end{equation}
It follows from (\ref{3.0}) and (\ref{3.0s}) that all zeros of $f(z)$ uniquely determine $f(z)$. The proof is complete.
\end{proof}

In the following theorem, under the assumption that  $0$ is not an eigenvalue, we prove a uniqueness result for the even inverse problem of the problem $L(q,\alpha,\beta,\alpha,\beta)$.
\begin{theorem}
    Assume that $\Delta_+\left(0\right)\neq0$, $q\left(x\right)=q\left(a-x\right)$, $\alpha=\alpha_{0}$ and $\beta=\beta_{0}$. Then  all the eigenvalues $\left\{\lambda_{k}^+\right\}$ (including multiplicity) uniquely determine $\alpha$, $\beta$ and $q(x)$ a.e. on $(0,a)$.
\end{theorem}
\begin{proof}
    Using Theorem \ref{th1} and $\alpha=\alpha_{0}$, we know that $\frac{2\alpha}{1+\alpha^2}$ is uniquely determined. Indeed, we can first determine whether $\alpha=1$ or not, by observing the number and the imaginary parts of  the eigenvalues $\left\{\lambda_{k}^+\right\}$. If $\alpha\ne1$, then we can use the asymptotics of the eigenvalues to recover $\frac{2\alpha}{1+\alpha^2}$.  Using Lemma \ref{hx11} and \eqref{c16}, whether $\alpha=1$ or not, we can obtain that the function $\frac{\Delta_+(\lambda)}{1+\alpha^2}$ is uniquely determined from its zeros. Using \eqref{k9}, \eqref{e15} and \eqref{o30}, we have
    \begin{equation}\label{o31}
        \frac{\Delta_+(\lambda)}{1+\alpha^2}-\frac{\Delta_{-}(\lambda)}{1+\alpha^2}=\frac{2i\alpha\lambda\Delta_{0}(\lambda)}{1+\alpha^2}.
    \end{equation}
It is known \cite{X}  that all zeros of $\Delta_{0}(\lambda)$ uniquely determine $q(x)$, $\alpha$ and $\beta$ under the assumption $\alpha\ne 1$. When $\alpha=1$, one can use a similar argument in \cite[Theorem 3.1]{XP} to prove the uniqueness theorem. 
Thus, it is enough to prove $\frac{\Delta_+(\lambda)}{1+\alpha^2}$ uniquely determines $\frac{\Delta_{-}(\lambda)}{1+\alpha^2}$.
    Define
    \begin{equation*}
        y_{a}\left(\lambda,x\right):=c\left(\lambda,a-x\right)-i\alpha\lambda s\left(\lambda,a-x\right).
    \end{equation*}
 Obviously, it satisfies
    \begin{align}
    \notag
     y_{a}\left(\lambda,x\right)=  y(-\lambda,a-x),\quad  y_{a}\left(\pm \lambda,a\right)=1,\quad y'_{a}\left(\pm \lambda,a\right)=\pm i\alpha\lambda-\beta.
    \end{align}
Since $q(x)=q(a-x)$, $\alpha=\alpha_0$ and $\beta=\beta_0$, using \eqref{e15}, we have
    \begin{equation}\label{e8}
        \Delta_{-}\left(\lambda\right)=-\left<y\left(\lambda,x\right),y_{a}\left(\lambda,x\right)\right>.
    \end{equation}
    Letting $x=\frac{a}{2}$ in \eqref{e8}, then we have
    \begin{align}
    \notag
        \Delta_{-}\left(\lambda\right)&=-y\left(\lambda,\frac{a}{2}\right)y'_{a}\left(\lambda,\frac{a}{2}\right)+y_{a}\left(\lambda,\frac{a}{2}\right)y'\left(\lambda,\frac{a}{2}\right)\\
        &\notag=-y_{a}\left(-\lambda,\frac{a}{2}\right)y'_{a}\left(\lambda,\frac{a}{2}\right)-y_{a}\left(\lambda,\frac{a}{2}\right)y_{a}'\left(-\lambda,\frac{a}{2}\right).
    \end{align}
    It follows that $\Delta_{-}\left(-\lambda\right)=\Delta_{-}\left(\lambda\right)$. Using \eqref{e7}, we have
    \begin{equation}
        \frac{\Delta_{-}(\lambda)}{1+\alpha^2}=\pm\sqrt{\frac{\Delta_+(\lambda)\Delta_+(-\lambda)}{(1+\alpha^2)^2}-\frac{4\alpha^2\lambda^2}{(1+\alpha^2)^2}}
    \end{equation}
    To determine which branch is the right choose, we note that $\frac{\Delta_{-}(0)}{1+\alpha^2}=\frac{\Delta_+(0)}{1+\alpha^2}\neq0$ that follows from \eqref{k9} and \eqref{e15}. Thus, $\frac{\Delta_{-}(\lambda)}{1+\alpha^2}$ is uniquely determined by $\frac{\Delta_+(\lambda)}{1+\alpha^2}$.
    The proof is complete.
\end{proof}
\begin{remark}
If $\Delta_+(0)=0$, then the uniqueness may not hold. It was known \cite{EK,Zo1} that for the problem $L(q,1,0,1,0)$, in the case $\Delta_+(0)=0$, there exist two even potentials corresponding to the same set of eigenvalues.
\end{remark}

Now, let us prove the so-called two-spectra theorem. Note that $\Delta_+(\lambda)$ and $\Delta_-(\lambda)$ have no common non-zero zeros.

\begin{theorem}\label{tht}
If $\alpha_0\ne1$ (or $\alpha\ne1$) and is known a priori, then all zeros of $\Delta_+(\lambda)$ and all non-zero  zeros of $\Delta_-(\lambda)$ (including multiplicity) uniquely determine $\alpha$ ($\alpha_0$), $\beta_0$, $\beta$ and  $q(x)$ a.e. on $(0,a)$.
\end{theorem}
\begin{proof}
Since $\alpha_0\ne1$ (or $\alpha\ne1$) and is known a priori, using Theorem \ref{th1}, we can determine the sign of $\left(\alpha_{0}-1\right)\left(1-\alpha\right)$. Indeed, if there are only finitely many eigenvalues $\{\lambda_n^+\}$ or infinitely many eigenvalues $\{\lambda_n^+\}$ with unbounded imaginary parts, then $\left(\alpha_{0}-1\right)\left(1-\alpha\right)=0$. If there are infinitely many eigenvalues $\{\lambda_n^+\}$ with bounded imaginary parts, then $\left(\alpha_{0}-1\right)\left(1-\alpha\right)\ne0$. In the latter case, if $|\sin (a{\rm Re}\lambda_n^+)|\to 1$ as $n\to\infty $ then $\left(\alpha_{0}-1\right)\left(1-\alpha\right)>0$; if $|\sin (a{\rm Re} \lambda_n^+)|\to 0$ as $n\to\infty $ then $\left(\alpha_{0}-1\right)\left(1-\alpha\right)<0$.  In particular, we can determine whether $\alpha=1$ ($\alpha_0=1$) or not.  If $\alpha\neq1$ ($\alpha_0\neq1$),  then we can recover the term $\frac{\alpha+\alpha_0}{1+\alpha\alpha_0}$ from the the asymptotics of the eigenvalues, and so $\alpha$ ($\alpha_0$) is uniquely determined. Using Lemma \ref{hx11} and \eqref{c16}, we can obtain that the function $\Delta_+(\lambda)$ is uniquely determined from its zeros. If $\alpha=1$, since $\alpha_0$ is known, by using Lemma \ref{hx11} and \eqref{c16}, we can still obtain that the function $\Delta_+(\lambda)$ is uniquely determined from its zeros. Using \eqref{e7}, it is easy to get whether $\lambda=0$ is a zero of $\Delta_-(\lambda)$. If $\Delta_-(0)=0$, then the multiplicity of $\lambda=0$ is also uniquely determined by using \eqref{e7}. Then, together with the condition of this theorem, we know that all zeros of $\Delta_-(\lambda)$ are known. Then using Lemma \ref{hx11} and \eqref{c16}, we know that $\Delta_-(\lambda)$ is uniquely determined by its zeros.
Using \eqref{k9}, \eqref{e15} and \eqref{o30}, we have
    \begin{equation}\label{o31s}
      {\Delta_+(\lambda)}+{\Delta_{-}(\lambda)}=y'\left(\lambda,a\right)+\beta y\left(\lambda,a\right).
    \end{equation}
    The right-hand side of \eqref{o31s} is the characteristic function of the problem \eqref{hx12}, \eqref{hx14} and $y'(0)-\beta y(0)=0$.
   Using Theorem 3.1 in \cite{XP}, we complete the proof.
\end{proof}

\begin{remark}(1) If $q(x)$ is real valued and $\beta,\beta_0\in \mathbb{R}$, then we can use the signs of imaginary parts of  zeros of $\Delta_-(\lambda)$ instead of the zeros of $\Delta_-(\lambda)$. Indeed, after obtaining $\Delta_+(\lambda)$ and $\alpha\alpha_0$,  the function $g(\lambda):=\Delta_-(\lambda)\Delta_-(-\lambda)$ is determined (cf.\eqref{e7}). 
Let $\{\xi_n\}$ (counted with multiplicities) be the zeros of $g(\lambda)$ in $\overline{\mathbb{C}_+}:=\{\lambda: {\rm Im}\lambda\ge0\}$. Then using Theorem \ref{th2x} (2)(ii), we know that $\xi_n$ or $\bar{\xi}_n$ is a zero of $\Delta_-(\lambda)$. Define
\begin{equation*}
\sigma_n={\rm sgn } ({\rm  Im} \lambda_n^-)=\left\{ \begin{split}
    +1,\quad &\text{if} \quad \lambda_n^-= \xi_n\in \mathbb{C}_+, \\
     -1,\quad &\text{if} \quad  \lambda_n^-=\bar{\xi}_n\in \mathbb{C}_-,\\
     0,\quad &\text{if} \quad  \lambda_n^-={\xi}_n\in \mathbb{R}.
 \end{split}\right.
\end{equation*}
Using the set of signs $\{\sigma_n\}$ and the zeros of $g(\lambda)$ in $\overline{\mathbb{C}_+}$, we can uniquely recovered all zeros of $\Delta_-(\lambda)$. We admit that this property was first observed by Korotyaev \cite{EK} in studying the inverse resonance problem $L(q,1,0,1,0)$. By contrast, when $\alpha\ne1$ and $\alpha_0\ne1$,  it is possible that all zeros of $\Delta_+(\lambda)$ together with at most  a finite number of the signs $\{\sigma_n\}$  uniquely recover all zeros of $\Delta_-(\lambda)$. For example, if $\alpha$ and $\alpha_0$ are determined such that $P_0^->0$, then from the asymptotics of $\lambda_n^-$ we know that there exist only finitely many $\lambda_n^-$s
in $\overline{\mathbb{C}_-}$. So, in this case,  there exist at most finitely many elements in $\{\xi_n\}$ belonging to the set of zeros of $\Delta_-(-\lambda)$ and the other elements are zeros of $\Delta_-(\lambda)$. That is to say, in this case, we only need at most  a finite number of signs to distinguish the zeros of $\Delta_-(-\lambda)$ and the zeros of $\Delta_-(\lambda)$ from all zeros of $g(\lambda)$ in $\overline{\mathbb{C}_+}$, and so all zeros of $\Delta_-(\lambda)$ are determined.

(2) If the known $\alpha_0$ is equal to $1$, it is unclear that if $\alpha$ can be determined. If $\alpha\ne1$ and is given, then it is included in Theorem \ref{tht}. If $\alpha_0=1=\alpha$, then $\Delta_+(\lambda)$ is still uniquely determined from its zeros by using Lemma \ref{hx11} and \eqref{c16}. Whereas, $\Delta_-(\lambda)$ is uniquely determined from its zeros provided that there is a non-zero $b_0\in\{\beta,\beta_0,\beta+\beta_0,\beta-\beta_0\}$ and $b_0$ is known. If $\beta=\beta_0=0$, i.e., in the case of the resonance problem $L(q,1,0,1,0)$, all zeros of $\Delta_-(\lambda)$ cannot uniquely determine $\Delta_-(\lambda)$ unless an additional sign is given (see \cite{EK}).

(3) If $(\alpha_0-1)(\alpha-1)\ne0$ and the sign of either  $\alpha_0-1$ or $\alpha-1$ is known a priori, then $\alpha$ and $\alpha_0$ can all be uniquely recovered from the asymptotics of the eigenvalues $\{\lambda_k^\pm\}$. Indeed, we can first recover two numbers $P_0^\pm$ and then it follows from \eqref{k10} that
\begin{equation*}
e^{P_0^++P_0^-}=\frac{(\alpha_0+1)^2}{(\alpha_0-1)^2},\quad e^{P_0^+-P_0^-}=\frac{(\alpha+1)^2}{(\alpha-1)^2}.
\end{equation*}
If the  sign of $\alpha_0-1$ (or $\alpha-1$) is known, then $\alpha_0$ ($\alpha$) is obtained uniquely. Substituting $\alpha_0$ (or $\alpha$) into $P_0^+$, we get  $\alpha$ ($\alpha_0$).
\end{remark}

Next, let us give two uniqueness theorems for the general partial inverse problem, namely, the potential function is known in a subinterval $(b,a)$ with arbitrary $b\in (0,a)$.

\begin{lemma}[\textup{see  \cite[p.173]{B}}]\label{le1}
    For any entire function $f(z)\not\equiv0$ of exponential type, the following inequality holds true:
    \begin{equation}
        \mathop {\varliminf }\limits_{r \to \infty}  \frac{n_f(r)}{r}\leq\frac{1}{2\pi}\int_0^{2\pi}h_f(\theta)d\theta,
    \end{equation}
    where $n_f(r)$ denotes the number of zeros of $f(z)$ in the disk $|z|\le r$, and  $h_f(\theta)$ is the indicator function of $f(z)$ that is defined by
  \begin{equation*}
    h_f(\theta):=\mathop {\varlimsup }\limits_{r \to \infty }\frac{\ln |f(re^{i\theta})|}{r}.
  \end{equation*}
\end{lemma}

From Theorem \ref{th1}, we see that
\begin{equation*}
  n_{\Delta_\pm}(r)=\frac{2tr}{\pi}(1+o(1)),\quad r\to\infty,
\end{equation*}
where $t= a$ if $(\alpha_0-1)(1-\alpha)\ne0$ and $t\le a$ if $(\alpha_0-1)(1-\alpha)=0$.
 Let $\Omega_\pm$  be the sets of the zeros of the functions $\Delta_\pm(\lambda)$, respectively.
Denote $\Omega_0:=\Omega_+\cup\Omega_-$. Let $ n_{\Omega_0}(r)$ be the number of the points in $\Omega_0\cap \{\lambda:|\lambda|\le r\}.$ 

\begin{theorem}\label{th6}
    Assume that $q(x)$ is known a priori a.e. on $\left(b,a\right)$ with $b\in (0,a)$ and $\alpha$ and $\beta$ are given. If the subset $\Omega_0$ satisfies
    \begin{equation}
         n_{\Omega_0}(r)=\frac{2mr}{\pi}[1+o(1)],\quad r \to \infty,\quad  m>2b,
    \end{equation}
    then $\Omega_0$ uniquely determines $\alpha_{0}$ and $\beta_{0}$ and $q(x)$ a.e. on $\left(0,b\right)$.
\end{theorem}
\begin{proof}
    Suppose that there are two problems $L(q,\alpha_0,\beta_0,\alpha,\beta)$ and $L(\tilde{q},\tilde{\alpha}_0,\tilde{\beta}_0,\tilde{\alpha},\tilde{\beta})$, which satisfy that  $q(x)=\tilde{q}(x)$ a.e. on $\left(b,a\right)$ and $\alpha=\tilde{\alpha}$ and $\beta=\tilde{\beta}$. Let us prove $q(x)=\tilde{q}(x)$ on $\left(0,b\right)$ and $\alpha_{0}=\tilde{\alpha}_{0}$ and $\beta_{0}=\tilde{\beta}_{0}$ if $\Omega_0=\tilde{\Omega}_0$. Define
    \begin{equation}\label{p1}
        F(\lambda)=y\left(\lambda,b\right)\tilde{y'}\left(\lambda,b\right)-\tilde{y}\left(\lambda,b\right)y'\left(\lambda,b\right)
    \end{equation}
   Combing \eqref{r2} and \eqref{k8}, then we have
    \begin{equation}\label{p9}
        \lvert F(\lambda)\rvert\leq C|\lambda|e^{2b\lvert {\rm Im}\lambda\rvert}, \quad \lvert\lambda\rvert\to\infty, \quad C>0.
    \end{equation}
    Let $\lambda=re^{i\theta}$, so $\lvert {\rm Im}\lambda\rvert=r\lvert\sin\theta\rvert.$ Then using \eqref{p9}, we have
    \begin{equation}
        h_{F}(\theta):=\varlimsup\limits_{r \to \infty}\frac{\ln|F(re^{i\theta})|}{r} \le 2b|\sin \theta|,
    \end{equation}
    which implies
    \begin{equation}
        \int_{0}^{2\pi}h_{F}(\theta)d\theta \le
  2b\int_{0}^{2\pi}|\sin \theta|d\theta =
  4b\int_{0}^{\pi}\sin \theta d\theta = 8b.
    \end{equation}
    Since $q(x)=\tilde{q}(x)$ a.e. on $\left(b,a\right)$ and $\alpha=\tilde{\alpha}$ and $\beta=\tilde{\beta}$, then $\phi(\lambda,b)=\tilde{\phi}(\lambda,b)$ for all $\lambda\in \mathbb{C}$.

   Note that $\phi(\lambda,x)$ meets the conditions $\phi(\lambda,a)=1$ and $\phi'(\lambda,a)=-(i\alpha\lambda+\beta)$, it is easy to get the function $\phi(-\lambda,x)$ to satisfy the conditions $\phi(-\lambda,a)=1$ and $\phi'(-\lambda,a)=(i\alpha\lambda-\beta)$. Then we can rewrite  \eqref{e15} as
\begin{equation}\label{o64}
    \Delta_{-}(\lambda)=\left<\phi(-\lambda,x),y(\lambda,x)\right>.
\end{equation}
    Letting $x=b$ in \eqref{e10} and \eqref{o64}, we get
    \begin{align}\label{w1}
        \Delta_\pm\left(\lambda\right)=y'\left(\lambda,b\right)\phi\left(\pm\lambda,b\right)-\phi'\left(\pm\lambda,b\right)y\left(\lambda,b\right),
    \end{align}
    \begin{equation}\label{w2}
        \tilde{\Delta}_\pm\left(\lambda\right)=\tilde{y}'\left(\lambda,b\right)\phi\left(\pm\lambda,b\right)-\phi'\left(\pm\lambda,b\right)\tilde{y}\left(\lambda,b\right).
    \end{equation}
    Multiply \eqref{w1} and \eqref{w2} by $\tilde{y}'\left(\lambda,b\right)$ and $y'\left(\lambda,b\right)$, respectively, take the difference and get
     \begin{equation}\label{w3}
        F\left(\lambda\right)=\frac{y'\left(\lambda,b\right)\tilde{\Delta}_\pm \left(\lambda\right)-\tilde{y}'\left(\lambda,b\right)\Delta_\pm \left(\lambda\right)}{\phi'\left(\pm \lambda,b\right)}.
    \end{equation}
     Multiply \eqref{w1} and \eqref{w2} by $\tilde{y}\left(\lambda,b\right)$ and $y\left(\lambda,b\right)$, respectively, take the difference and get
    \begin{equation}\label{w4}
         F\left(\lambda\right)=\frac{y\left(\lambda,b\right)\tilde{\Delta}_\pm \left(\lambda\right)-\tilde{y}\left(\lambda,b\right)\Delta_\pm \left(\lambda\right)}{\phi\left(\pm \lambda,b\right)}.
    \end{equation}
     Considering \eqref{w3} and \eqref{w4}, as $\phi\left(\pm \lambda,b\right)$ and $\phi'\left(\pm\lambda,b\right)$ cannot vanish simultaneously, then we conclude that all common zeros (including multiplicity) of $\Delta_\pm\left(\lambda\right)$ and $\tilde{\Delta}_\pm\left(\lambda\right)$ are zeros of $F\left(\lambda\right)$.
    Thus, we have
     \begin{equation}
         \varliminf\limits_{r \to \infty}\frac{n_{F}(r)}{r} \ge
  \lim_{r \to \infty}\frac{n_{\Omega_0}(r)}{r} =
  \frac{2m}{\pi}.
     \end{equation}
     Using Lemma \ref{le1}, if the entire function $F(\lambda)\not\equiv0$, then
     \begin{equation}
         \frac{2m}{\pi}\le  \varliminf\limits_{r \to \infty}\frac{n_{F}}{r} \le \frac{1}{2\pi}\int_0^{2\pi}h_{F}(\theta)d\theta \le
  \frac{4b}{\pi},
     \end{equation}
     which implies $m\leq 2b$. This contradicts $m>2b$. Hence, $F(\lambda)\equiv0$, namely, $\frac{y'\left(\lambda,b\right)}{y\left(\lambda,b\right)}=\frac{\tilde{y}'\left(\lambda,b\right)}{\tilde{y}\left(\lambda,b\right)}$. Since  $y'\left(\lambda,b\right)$ and $y\left(\lambda,b\right)$ do not have common zeros, the zeros of ${y\left(\lambda,b\right)}$ are uniquely determined.  Using Theorem 3.2 in \cite{X}, we obtain that $q(x)=\tilde{q}(x)$ a.e. on $\left(0,b\right)$, $\alpha_{0}=\tilde{\alpha}_{0}$ and $\beta_{0}=\tilde{\beta}_{0}$. The proof is complete.
\end{proof}

\begin{corollary}\label{th7}
    If $q(x)$ is known a priori a.e. on $\left(b,a\right)$ with $b<\frac{a}{2}$ and $\alpha$ and $\beta$ are known, then the subset $\Omega:=\Omega_+$ (or $\Omega:=\Omega_-$) satisfying
    \begin{equation}
         n_{\Omega}(r)=\frac{2mr}{\pi}[1+o(1)],\quad r \to \infty,\quad  2b<m\le a,
    \end{equation}
     uniquely determines $\alpha_{0}$, $\beta_{0}$ and $q(x)$ a.e. on $\left(0,b\right)$.
\end{corollary}

In the above theorem, it needs $m$ to be greater than $2b$. In the following theorem, we consider the critical case $m=2b$. In this case, we assume $(\alpha_0-1)(1- \alpha)\ne0$.
Denote
\begin{equation}\label{xk1}
  g_\pm (\lambda):=\lambda[i\left(\alpha_{0}\pm \alpha\right)\cos{\lambda a}-\left(1\pm\alpha\alpha_{0}\right)\sin{\lambda a}].
\end{equation}
Let  $\{\mu_k^\pm\}$ be zeros of $g_\pm(\lambda)$, respectively. It is obvious that
\begin{equation}\label{o28}
    \begin{cases}&  \!\!\!\!\mu_{0}^\pm=0 ,\;
        \mu_{k}^\pm=\frac{\pi}{a}\left(\lvert k\rvert-\frac{1}{2}\right){\rm sgn}k+\frac{iP_0^\pm}{2a},\; k\in\mathbb{Z}_0,\ \text{if}\ \left(\alpha_{0}-1\right)\left(1-\alpha\right)>0, \\
      & \!\! \!\!\!\!  \mu_{-1}^\pm=0,\;\mu_{k}^\pm=\frac{\pi}{a}\left(\lvert k\rvert-1\right){\rm sgn}k+\frac{iP_0^\pm}{2a}, k\in\mathbb{Z}_{0}\setminus\{-1\} ,\ \text{if}\ \left(\alpha_{0}-1\right)\left(1-\alpha\right)<0.
    \end{cases}
\end{equation}
From Theorem \ref{th1}, we see that
\begin{equation*}
  \lambda_k^\pm =\mu_k^\pm +O(1/k),\quad k\to\infty.
\end{equation*}
\begin{lemma}[\textup{\cite[Proposition B.6]{FB}}]\label{lem2}
    Assume that $E_0(\rho)$  is an  entire function of order less than one. If $\lim_{\lvert t\rvert\to\infty, t\in\mathbb{R}}E_0(it)=0$, then $E_0(\rho)\equiv0$ on the whole complex plane.
\end{lemma}

\begin{theorem}\label{thxx}
    Assume that $(\alpha_0-1)(1-\alpha)\ne0$, and $q(x)$ is known a priori a.e. on $(b,a)$ with $b\in (0,a)$ and $\alpha$ and $\beta$ are known. If the subset $\Omega_{1}:=\{\lambda_{k_{j}}^+\}_{j\in \mathbb{Z}_1}\cup \{\lambda_{k_{j}}^-\}_{j\in \mathbb{Z}_1}$ satisfies
    \begin{equation}\label{i1}
        \sum_{j\in \mathbb{Z}_1}\frac{\lvert\lambda_{k_{j}}^+-a\mu_{j}^+/b_+\rvert}{\lvert j\rvert+1}+\sum_{j\in \mathbb{Z}_1}\frac{\lvert\lambda_{k_{j}}^--a\mu_{j}^-/b_-\rvert}{\lvert j\rvert+1}<\infty,\quad b_++b_-=2b,
    \end{equation}
  where $\mathbb{Z}_1=\mathbb{Z}_0$ if $(\alpha_0-1)(1- \alpha)<0$ and $\mathbb{Z}_1=\mathbb{Z}$ if $(\alpha_0-1)(1- \alpha)>0$,  then $\Omega_{1}$ uniquely determines $\alpha_{0},\beta_{0}$ and $q(x)$ a.e. on $(0,b)$.
\end{theorem}
\begin{proof}
   We only deal with the case $(\alpha_0-1)(1- \alpha)>0$, i.e., $\mathbb{Z}_1=\mathbb{Z}$. Let $\rho=\lambda^{2}$. 
   Define
    \begin{equation}\label{p2}
        G(\rho)=F(\lambda)F(-\lambda),\quad
        \Phi(\rho)=\prod_{j\in \mathbb{Z}}\left(1-\frac{\rho}{(\lambda_{k_{j}}^+)^{2}}\right)\left(1-\frac{\rho}{(\lambda_{k_{j}}^-)^{2}}\right),
  \quad
        E_0(\rho)=\frac{G(\rho)}{\Phi(\rho)},
    \end{equation}
    where $F$ is defined in \eqref{p1}. Since $F(\lambda)$ is an entire function of order $\le1$ and $F(\lambda)F(-\lambda)$ is an even function of $\lambda$, then $G(\rho)$  is an entire function of $\rho $ of  order $\le\frac{1}{2}$.  Let us show that $E_0(\rho)$  is also an entire function of $\rho$ of order $\le\frac{1}{2}$.  Since all zeros of $\Phi(\rho) $ are zeros of $G(\rho)$, it is enough to show that $\Phi(\rho) $ is an entire function of $\rho $ of  order $\le\frac{1}{2}$.
    By virtue of \eqref{k11} and \eqref{k12}, we have
    \begin{equation}\label{4.26}
        \frac{1}{(\lambda_{k_{j}}^\pm)^{2}}=O\left(\frac{1}{j^2}\right), \quad j\to\infty.
    \end{equation}
    It follows that the series
    \begin{equation}
    \notag
        \sum_{j\in \mathbb{Z}}\left(\left\lvert \frac{\rho}{(\lambda_{k_{j}}^+)^{2}} \right\rvert+\left\lvert \frac{\rho}{(\lambda_{k_{j}}^-)^{2}}\right\rvert\right)
    \end{equation}
    converges uniformly on bounded subsets of $\mathbb{C}$. Consequently, the infinite product $\Phi(\rho)$ in \eqref{p2} converges to an entire function of $\rho$, with its roots being exactly $(\lambda_{k_{j}}^+)^{2}$ and $(\lambda_{k_{j}}^-)^{2}$, $j\in \mathbb{Z}$. Denote
    \begin{equation}
        \Xi_\Phi:=\inf \left\{p: \sum_{j\in \mathbb{Z}}\left(\frac{1}{\left \lvert(\lambda_{k_{j}}^+)^{2p}\right\rvert} +\frac{1}{\left\lvert (\lambda_{k_{j}}^-)^{2p}\right\rvert}\right)<\infty\right\},
    \end{equation}
    which is called \emph{convergence exponent of zeros} of the canonical product of $\Phi(\rho)$ in \eqref{p2}. Obviously,  the estimate  \eqref{4.26} implies $\Xi_\Phi\le\frac{1}{2}$. Note that the order of canonical product of an entire function coincides with its convergence exponent of zeros (see \cite[p.16]{B}). It follows  that the order of canonical
product of $\Phi(\rho)$
 is less than or equal to $1/2$. Using Hadamard's factorization theorem, we know that the infinite product in \eqref{p2} is
the canonical product of the function
$\Phi(\rho)$, and so the order of $\Phi(\rho)$
 is at most $1/2$.

    Using the Lemma \ref{lem2} and the end part of proof of Theorem \ref{th6}, to prove Theorem \ref{thxx}, it is enough to prove
    \begin{equation}
        \lim_{\lvert t\rvert\to\infty,t\in\mathbb{R}}E_0(it)=0.
    \end{equation}
        Define
        \begin{equation}
            \Phi_{0}\left(\rho\right)=g_+\left(\frac{b_+\lambda}{a}\right)g_+\left(-\frac{b_+\lambda}{a}\right)g_-\left(\frac{b_-\lambda}{a}\right)
            g_-\left(-\frac{b_-\lambda}{a}\right),
        \end{equation}
        where the functions $g_\pm$ are defined in \eqref{xk1}.
        Using the condition $d_++d_-=2b$, we have that for large $|t|$, there holds
        \begin{equation}\label{xk4}
            \lvert\Phi_{0}\left(it\right)\rvert\geq C\lvert t\rvert^{2}e^{4b\lvert {\rm Im} \sqrt{t}\rvert},\quad C>0.
        \end{equation}
      Using the Hadamard's factorization theorem, we can also have
        \begin{equation}
            \Phi_{0}\left(\rho\right)=C_1\rho^2\prod_{j\in \mathbb{Z}_0}\left(1-\frac{\rho}{(\zeta_{j}^+)^{2}}\right)\left(1-\frac{\rho}{(\zeta_{j}^-)^{2}}\right),
        \end{equation}
        where $\zeta_{j}^\pm=\frac{a}{b_\pm}\mu_{j}^\pm$ and $C_1$ is a constant. 
         Then
        \begin{align*}
            \left |\frac{\Phi_{0}\left(it\right)}{\Phi\left(it\right)}\right|\notag&\le C\prod_{j\in \mathbb{Z}_0}\frac{\left(1-\frac{it}{(\zeta_{j}^+)^{2}}\right)\left(1-\frac{it}{(\zeta_{j}^-)^{2}}\right)}{\left(1-\frac{it}{(\lambda_{k_{j}}^+)^{2}}\right)\left(1-\frac{it}{(\lambda_{k_{j}}^-)^{2}}\right)}\\
            &\leq C\!\!\prod_{j\in \mathbb{Z}_0}\!\!\left[1+\left|\frac{(\zeta_{j}^+)^{2}-(\lambda_{k_{j}}^+)^{2}}{(\lambda_{k_{j}}^+)^{2}-it}\right|\right]\!\!\left[1+\left|\frac{(\zeta_{j}^-)^{2}-(\lambda_{k_{j}}^-)^{2}}{(\lambda_{k_{j}}^-)^{2}-it}\right|\right]\prod_{j\in \mathbb{Z}_0}\left|\frac{(\lambda_{k_{j}}^+)^{2}}{(\zeta_{j}^+)^{2}}\right|\left|\frac{(\lambda_{k_{j}}^-)^{2}}{(\zeta_{j}^-)^{2}}\right|,
        \end{align*}
    which implies from \eqref{i1} that
\begin{equation}\label{xk6}
    \left |\frac{\Phi_{0}\left(it\right)}{\Phi\left(it\right)}\right|=O(1), \quad\lvert t\rvert\to  \infty.
\end{equation}
Using \eqref{p9}, \eqref{p2}, \eqref{xk4} and \eqref{xk6}, we get
\begin{equation}
    \lvert E_0\left(it\right)\rvert=\left|\frac{G(it)\Phi_{0}(it)}{\Phi_{0}(it)\Phi(it)}\right|=O\left(\frac{1}{\lvert t\rvert}\right), \quad\lvert t\rvert\to \infty.
\end{equation}.\\
The proof is complete.
\end{proof}
\begin{corollary}\label{cor2}
   Assume that $(\alpha_0-1)(1-\alpha)\ne0$. If $q(x)$ is known a priori a.e. on $(b,a)$ with $b<\frac{a}{2}$, and $\alpha$,$\beta$ are known, then the subset  $\{\lambda_{k_{j}}^+\}_{j\in \mathbb{Z}_1}$ (or $ \{\lambda_{k_{j}}^-\}_{j\in \mathbb{Z}_1}$) satisfying
     \begin{equation*}
        \sum_{j\in \mathbb{Z}_1}\frac{\lvert\lambda_{k_{j}}^+-\frac{a\mu_{j}^+}{2b}\rvert}{\lvert j\rvert+1}<\infty,\quad \left(\text{or}\quad \sum_{j\in \mathbb{Z}_1}\frac{\lvert\lambda_{k_{j}}^--\frac{a\mu_{j}^-}{2b}\rvert}{\lvert j\rvert+1}<\infty,\right)
    \end{equation*}
     uniquely determines $\alpha_{0}$, $\beta_{0}$ and $q(x)$ a.e. on $\left(0,b\right)$.
\end{corollary}

\noindent\textbf{\large Appendix}
\\[-2mm]

\appendix
\setcounter{equation}{0}
\renewcommand\theequation{A.\arabic{equation}}
\setcounter{lemma}{0}
\renewcommand\thelemma{A.\arabic{lemma}}

In Appendix, we give the asymptotics of zeros of the entire function
    \begin{equation}\label{A.7}
        \varphi(\lambda)=\phi^{(1)}(\lambda)+M \sin{\lambda a}-iN \cos{\lambda a}+\psi(\lambda),
    \end{equation}
where $M,\ N\in\mathbb{C}$, $\psi$ is an entire function of exponential type $\leq a$ and belongs to $L^2(-\infty,\infty)$ and
    \begin{equation}\label{A.1}
        \phi^{(1)}(\lambda)=\lambda(\sigma_1\cos{\lambda a}+i\sigma_2\sin{\lambda a}),
    \end{equation}
with $\sigma_1,\sigma_2\in \mathbb{R}$. In \cite[p.178]{MV}, the asymptotics of zeros of $  \varphi(\lambda)$ is given for the condition $\sigma_1\ne\sigma_2$ and $\sigma_1,\sigma_2\ge0$. Indeed, this condition can be weaker. Let us first prove the following lemma.

\begin{lemma}\label{LA1}
    Let $\sigma_1$, $\sigma_2$ be real numbers with $\sigma_1\ne\pm\sigma_2$. Then the zeros of the entire function $\phi^{(1)}(\lambda)$  defined in \eqref{A.1}
    have the following asymptotic representations.\\
    $\left(1\right)$ If $\left(\sigma_2-\sigma_1\right)\left(\sigma_2+\sigma_1\right)<0$, then the zeros $\{\lambda_{k}^{(1)}\}_{k\in\mathbb{Z}}$ of $\phi^{(1)}(\lambda)$ are $\lambda_{0}^{(1)}=0$,
   \begin{equation}\label{XA1}
   \begin{aligned}
      \lambda_{k}^{(1)}=\frac{\pi}{a}\left(\lvert k\rvert-\frac{1}{2}\right){\rm sgn} k+\frac{i}{2a}\ln{\left(\frac{|\sigma_2+\sigma_1|}{|\sigma_2-\sigma_1|}\right)}, \quad k\in\mathbb{Z}_0,\quad \mathbb{Z}_0=\mathbb{Z}\setminus\{0\},
   \end{aligned}
   \end{equation}
   $\left(2\right)$ If $\left(\sigma_2-\sigma_1\right)\left(\sigma_2+\sigma_1\right)>0$, then the zeros $\{\lambda_{k}^{(1)}\}_{k\in \mathbb{Z}_0}$ of $\phi^{(1)}(\lambda)$ are $\lambda_{-1}^{(1)}=0$:
   \begin{equation}\label{XA2}
   \begin{aligned}
      \lambda_{k}^{(1)}=\frac{\pi}{a}\left(\lvert k\rvert-1\right){\rm sgn} k+\frac{i}{2a}\ln{\left(\frac{|\sigma_2+\sigma_1|}{|\sigma_2-\sigma_1|}\right)}, \quad k\in \mathbb{Z}_0\setminus\{-1\}.
   \end{aligned}
   \end{equation}
\end{lemma}
\begin{proof}
Note that
\begin{equation}\label{A7}
    \cos{\lambda a}=\frac{e^{i\lambda a}+e^{-i\lambda a}}{2}, \quad \sin{\lambda a}=\frac{e^{i\lambda a}-e^{-i\lambda a}}{2i}.
\end{equation}
Then $\frac{1}{\lambda}\phi^{(1)}(\lambda)=0$ is equivalent to
\begin{equation}\label{A.5}
    \left(\sigma_2+\sigma_1\right)e^{i\lambda a}=\left(\sigma_2-\sigma_1\right)e^{-i\lambda a}.
\end{equation}
Since $\sigma_1\ne\pm \sigma_2$, it follows that
\begin{equation}\label{A11}
    e^{2i\lambda a}=\frac{\sigma_2-\sigma_1}{\sigma_2+\sigma_1}.
\end{equation}
Taking the logarithm on both sides of \eqref{A11}, we have
\begin{equation*}
  2i\lambda_k^{(1)} a=\ln{\left(\frac{\sigma_2-\sigma_1}{\sigma_2+\sigma_1}\right)}+2k\pi i,\quad k\in \mathbb{Z}.
\end{equation*}
Note that $\ln z=\ln |z|+i\arg z$ with $\arg z\in (-\pi,\pi]$. It follows that
\begin{equation}
    \lambda_{k}^{(1)}=\begin{cases}&  \!\!\!\!
        \frac{k\pi}{a}+\frac{i}{2a}\ln{\left(\frac{\sigma_2+\sigma_1}{\sigma_2-\sigma_1}\right)},\; k\in\mathbb{Z},\ \text{if}\ \frac{\sigma_2+\sigma_1}{\sigma_2-\sigma_1}>0, \\
      & \!\! \!\!\!\!  \frac{\pi}{a}\left(k-\frac{1}{2}\right)+\frac{i}{2a}\ln{\left(\frac{|\sigma_2+\sigma_1|}{|\sigma_2-\sigma_1|}\right)},\; k\in\mathbb{Z},\ \text{if}\ \frac{\sigma_2+\sigma_1}{\sigma_2-\sigma_1}<0,
    \end{cases}
\end{equation}
which is equivalent to \eqref{XA1} and \eqref{XA2}. The proof is complete.
\end{proof}
Using Lemma \ref{LA1} and a similar argument to the proof of Lemma 7.1.3 in \cite{MV}, one can get the following lemma.
\begin{lemma}\label{LA2}
Let $\sigma_1$ and $\sigma_2$ be real  numbers with $\sigma_1\ne\pm\sigma_2$, and the entire function $ \varphi(\lambda)$ is defined in \eqref{A.7}.\\
$\left(1\right)$ If $\left(\sigma_2-\sigma_1\right)\left(\sigma_2+\sigma_1\right)<0$, then the zeros $\{\lambda_{k}\}_{k\in\mathbb{Z}}$ of $\varphi(\lambda)$ have the following asymptotic behavior:
   \begin{equation}
   \begin{aligned}
      \lambda_{k}=\frac{\pi}{a}\left(\lvert k\rvert-\frac{1}{2}\right){\rm sgn} k+\frac{i}{2a}\ln{\left(\frac{|\sigma_2+\sigma_1|}{|\sigma_2-\sigma_1|}\right)}+\frac{P}{k}+\frac{\gamma_k}{k}, \quad  \lvert k\rvert\to\infty,
   \end{aligned}
   \end{equation}
   where $\{\gamma_{k}\}_{k=-\infty,k\neq0}^{\infty}\in l_{2}$ and
   \begin{equation}\label{A.9}
       P=\frac{\sigma_2 N-\sigma_1 M}{\pi(\sigma_2^2-\sigma_1^2)}.
   \end{equation}
   $\left(2\right)$ If $\left(\sigma_2-\sigma_1\right)\left(\sigma_2+\sigma_1\right)>0$, then the zeros $\{\lambda_{k}\}_{k\in\mathbb{Z}_0}$ of $\varphi(\lambda)$ have the following asymptotic behavior:
   \begin{equation}
   \begin{aligned}
      \lambda_{k}=\frac{\pi}{a}\left(\lvert k\rvert-1\right){\rm sgn} k+\frac{i}{2a}\ln{\left(\frac{|\sigma_2+\sigma_1|}{|\sigma_2-\sigma_1|}\right)}+\frac{P}{k}+\frac{\gamma_k}{k}, \quad  \lvert k\rvert\to\infty.
   \end{aligned}
   \end{equation}
  where $P$ is given \eqref{A.9} and $\{\gamma_{k}\}_{k=-\infty,k\neq0}^{\infty}\in l_{2}$.
\end{lemma}
\begin{proof}
The proof is similar to the proof of Lemma 7.1.3 in \cite{MV}. But for the  convenience of readers, we just give the main idea of the proof.

Denote $G_\delta:=\{\lambda\in\mathbb{C}: \lvert \lambda-\lambda_{k}^{(1)}\rvert\geq\delta>0, k\in\mathbb{Z}_0\}$.
Clearly, there exist positive constants $C_\delta$ and $K$ such that
\begin{equation}\label{A10}
    \lvert\phi^{(1)}(\lambda)\rvert\geq C_\delta\lvert \lambda \rvert e^{a\lvert{\rm Im}\lambda \rvert}, \quad \lambda \in G_\delta,\quad \lvert \lambda \rvert\geq \lambda^* ,
\end{equation}
for sufficiently large $\lambda^*=\lambda^*(\delta)$, and
\begin{equation}\label{A.12}
    \lvert \varphi(\lambda)-\phi^{(1)}(\lambda)\rvert\le Ke^{a\lvert{\rm Im}\lambda \rvert}.
\end{equation}
It follows  that
\begin{equation}
    \lvert\phi^{(1)}(\lambda)\rvert>\lvert \varphi(\lambda)-\phi^{(1)}(\lambda)\rvert,\quad \lambda\in G_\delta\cap \{\lambda:|\lambda|\ge \lambda^*\},
\end{equation}
 for sufficiently large $\lambda^*$. Therefore, there exist sufficient large numbers $R_k$ and sufficiently small numbers $\delta_k>0$ such that
 \begin{equation}
  \varphi(\lambda)\ne0, \quad \lvert \phi^{(1)}(\lambda)\rvert>\lvert \varphi(\lambda)-\phi^{(1)}(\lambda)\rvert,\quad \lambda\in\Gamma_{R_k}\cup \gamma_{\delta_k},
\end{equation}
where
\begin{equation}
    \Gamma_{R_k}:=\{\lambda\in\mathbb{C}: \lvert \lambda \rvert=R_k\},\quad \gamma_{\delta_k}:=\{\lambda\in\mathbb{C}: \lvert \lambda-\lambda_{k}^{(1)}\rvert=\delta_k\}.
\end{equation}
Using  Rouch\`{e}'s theorem, we conclude that $\varphi(\lambda)$ and $\phi^{(1)}(\lambda)$ have the same number of zeros, counted with multiplicity, inside $\Gamma_{R_k}$ and $\gamma_{\delta_k}$. In particular,  in $\gamma_{\delta_k}$ there is exactly one zero of $\varphi(\lambda)$. It follows
that
\begin{equation}\label{A12}
    \lambda_k=\lambda_{k}^{(1)}+\varepsilon_k,\quad \varepsilon_k=o(1),\quad |k|\to\infty.
\end{equation}
Substituting \eqref{A12} into \eqref{A.7}, we get  for large $|k|$  that
    \begin{align}\label{XC0}
    \notag
         0=&\varphi(\lambda_k)=\phi^{(1)}(\lambda_{k}^{(1)}+\varepsilon_k)+M \sin{(\lambda_{k}^{(1)}+\varepsilon_k) a}-iN \cos{(\lambda_{k}^{(1)}+\varepsilon_k) a}+\psi(\lambda_k)\\
         \notag=& (\sigma_1\lambda_{k}-iN)\cos{(\lambda_{k}^{(1)}+\varepsilon_k) a}+(i\sigma_2\lambda_{k}+M)\sin{(\lambda_{k}^{(1)}+\varepsilon_k) a}+\psi(\lambda_k)\\
         \notag=&(\sigma_1\lambda_{k}-iN)\left[\cos{(\lambda_{k}^{(1)} a)}\cos{(\varepsilon_k a)}-\sin{(\lambda_{k}^{(1)} a)}\sin{(\varepsilon_k a)}\right]\\
         \notag &+(i\sigma_2\lambda_{k}+M)\left[\sin{(\lambda_{k}^{(1)} a)}\cos{(\varepsilon_k a)}+\cos{(\lambda_{k}^{(1)} a)}\sin{(\varepsilon_k a)}\right]+\psi(\lambda_k).\\
          \notag=&\psi(\lambda_k)+\frac{\lambda_k}{\lambda_k^{(1)}}\phi^{(1)}(\lambda_k^{(1)})\cos (\varepsilon_ka)+\lambda_k\left[-\sigma_1\sin{(\lambda_{k}^{(1)} a)}+i\sigma_2\cos{(\lambda_{k}^{(1)} a)}\right]\sin{(\varepsilon_k a)}\\
         &+\left[M\sin{(\lambda_{k}^{(1)} a)}-iN\cos{(\lambda_{k}^{(1)} a)}\right]\cos{(\varepsilon_k a)}+\left[M\cos{(\lambda_{k}^{(1)} a)}+iN\sin{(\lambda_{k}^{(1)} a)}\right]\sin{(\varepsilon_k a)}.
    \end{align}
    Using \eqref{A11} and \eqref{A7}, we have for large $|k|$ that
    \begin{equation}\label{XC}
   -\sigma_1\sin{(\lambda_{k}^{(1)} a)}+i\sigma_2\cos{(\lambda_{k}^{(1)} a)}=ie^{-i\lambda_k^{(1)}a}(\sigma_2-\sigma_1),
    \end{equation}
    \begin{equation}\label{XC1}
   M\sin{(\lambda_{k}^{(1)} a)}-iN\cos{(\lambda_{k}^{(1)} a)}=ie^{-i\lambda_k^{(1)}a}\frac{M\sigma_1-N\sigma_2}{\sigma_2+\sigma_1},
    \end{equation}
    \begin{equation}\label{XC2}
    M\cos{(\lambda_{k}^{(1)} a)}+iN\sin{(\lambda_{k}^{(1)} a)}=e^{-i\lambda_k^{(1)}a}\frac{M\sigma_2-N\sigma_1}{\sigma_2+\sigma_1}.
    \end{equation}
Substituting \eqref{XC}--\eqref{XC2} into \eqref{XC0}, and noting $ \phi^{(1)}(\lambda_{k}^{(1)})=0$, we have
\begin{equation}\label{XC3}
 \sin{(\varepsilon_k a)}=\frac{i\left(N\sigma_2-M\sigma_1\right)\cos(\varepsilon_ka)-\psi(\lambda_k)(\sigma_2+\sigma_1)e^{i\lambda_k^{(1)}a}}{i\lambda_k(\sigma_2^2-\sigma_1^2)+\left(M\sigma_2-N\sigma_1\right)} .
\end{equation}
Note that
\begin{equation}\label{XC4}
 \sin{\varepsilon_k a}=\varepsilon_k a+O(\lvert \varepsilon_k\rvert^{-3}), \ \cos{\varepsilon_k a}=1+O(\lvert \varepsilon_k\rvert^{-2}), \ \lambda_k=\frac{k\pi }{a}[1+O(1/k)], \ |k|\to \infty.
\end{equation}
Substituting \eqref{XC4} into \eqref{XC3}, and noting $\{\psi(\lambda_k)\}\in l_2$, we get
\begin{equation}
    \varepsilon_k=\frac{P}{k}+\frac{\gamma_k}{k}.
\end{equation}
The proof is complete.
\end{proof}

\end{document}